\documentclass[12pt]{elsarticle}
\usepackage{amsthm,amsmath,amssymb,mathabx,extsizes}

\newcommand{\FF}{\mathbb F}
\newcommand{\RR}{\mathbb R}
\newcommand{\mc}{\mathcal }

\newcommand{\ci}{
\begin{picture}(6,6)
\put(3,3.6){\circle*{3}}
\end{picture}}

\usepackage[active]{srcltx}

\DeclareMathOperator{\rank}{rank}

\newcommand{\matt}[1]{\left[\begin{smallmatrix}
   #1\end{smallmatrix}\right]}
\newcommand{\mat}[1]{\begin{bmatrix}
   #1\end{bmatrix}}

\newtheorem{theorem}{Theorem}

\theoremstyle{definition}

\begin{document}
\title{Wildness of the problems of classifying two-dimensional spaces of commuting linear operators and certain Lie algebras}

\author[fut]{Vyacheslav Futorny}
\ead{futorny@ime.usp.br}
\address[fut]{Department of Mathematics, University of S\~ao Paulo, Brazil}

\author[kly,pet]{Tetiana Klymchuk}
\ead{tetiana.klymchuk@upc.edu}
\address[kly]{Universitat Polit\`{e}cnica de Catalunya, Barcelona, Spain}

\author[pet]{Anatolii P.~Petravchuk}
\ead{aptr@univ.kiev.ua}
\address[pet]{Faculty of Mechanics and Mathematics, Taras Shevchenko University, Kiev, Ukraine}

\author[ser]{Vladimir~V.~Sergeichuk\corref{cor}}
\ead{sergeich@imath.kiev.ua}
\address[ser]{Institute of Mathematics,
Tereshchenkivska 3, Kiev, Ukraine}


\cortext[cor]{Corresponding author.  Published in: Linear Algebra Appl. 536 (2018) 201--209.}


\begin{abstract}
For each two-dimensional vector space $V$  of commuting $n\times n$ matrices over a field  $\FF$ with at least 3 elements, we denote by $\widetilde V$ the vector space of all $(n+1)\times(n+1)$ matrices of the form $\matt{A&*\\0&0}$ with $A\in V$. We prove the wildness of the problem of classifying  Lie algebras $\widetilde V$ with the bracket operation $[u,v]:=uv-vu$. We also prove the wildness of the problem of classifying two-dimensional vector spaces consisting of
commuting linear operators on a vector space
over a field.
 \end{abstract}


\maketitle

\section{Introduction}
\label{s_intr}

Let $\FF$ be a field that is not the field with 2 elements.
We prove the wildness of the problems of classifying
\begin{itemize}
  \item two-dimensional vector spaces consisting of
commuting linear operators on a vector space
over $\FF$ (see Section 2), and
  \item Lie algebras  $\text{\rm L}(V)$ with bracket $[u,v]:=uv-vu$ of matrices of the form
\begin{equation}\label{11q}
\begin{bmatrix}
  A & \begin{matrix}
        \alpha _1\\\vdots\\ \alpha _n \\
      \end{matrix}
   \\
  \begin{matrix}
        0&\dots&0 \\
      \end{matrix} & 0 \\
\end{bmatrix}, \qquad\text{in which } A\in V,\ \ \alpha _1,\dots,\alpha _n\in\FF,
\end{equation}
in which $V$ is any two-dimensional vector space of  $n\times n$ commuting matrices over $\FF$ (see Section 3).
\end{itemize}

A classification problem is called \emph{wild} if it contains the problem of classifying  pairs of $n\times n$ matrices up to similarity transformations
\begin{equation*}\label{cfa}
(M,N)\mapsto S^{-1}(M,N)S:=(S^{-1}MS,S^{-1}NS)
\end{equation*}
with nonsingular $S$. This notion was introduced by Donovan and Freislich \cite{don1,don2}. Each wild problem is considered as hopeless since it contains the problem of classifying an arbitrary system of linear mappings, that is, representations of an arbitrary quiver (see \cite{gel-pon,bel-ser_compl}).

Let $\mc U$ be an $n$-dimensional vector space over $\FF$. The problem of classifying linear operators $\mc A:\mc U\to \mc U$ is the problem of classifying matrices $A\in\FF^{n\times n}$ up to similarity transformations
$A\mapsto S^{-1}AS
$
with nonsingular $S\in \FF^{n\times n}$. In the same way, the problem of classifying vector spaces $\mc V$ of linear operators on  $\mc U$ is the problem of classifying matrix vector spaces $V\subset \FF^{n\times n}$ up to similarity transformations
\begin{equation}\label{qwa}
V\mapsto S^{-1}VS:=\{S^{-1}AS\,|\,A\in V\}
\end{equation}
with nonsingular $S\in \FF^{n\times n}$ (the spaces $V$ and $S^{-1}VS$ are \emph{matrix isomorphic}; see \cite{gro}).
In Theorem \ref{vtd}(a), we prove the wildness of the  problem of classifying two-dimensional vector spaces $V\subset \FF^{n\times n}$ of  commuting matrices up to transformations \eqref{qwa}.

Each two-dimensional vector space $V\subset \FF^{n\times n}$ is given by its basis $A,B\in V$ that is determined up to transformations
$(A,B)\mapsto(\alpha A+\beta B,\gamma A+\delta B)$, in which $\matt{\alpha &\gamma \\\beta &\delta }\in \FF^{2\times 2}$ is a change-of-basis matrix. Thus, the problem of classifying two-dimensional vector spaces $V\subset \FF^{n\times n}$ up to transformations \eqref{qwa} is the problem of classifying pairs of linear independent matrices $A,B\in \FF^{n\times n}$ up to transformations
\begin{equation}\label{rtf}
(A,B)\mapsto (A',B'):= S^{-1}(\alpha A+\beta B,\gamma A+\delta B)S,
\end{equation}
in which both $S\in \FF^{n\times n}$ and $\matt{\alpha &\beta \\\gamma &\delta }\in \FF^{2\times 2}$ are nonsingular matrices. We say that the matrix pairs $(A,B)$ and $(A',B')$ from \eqref{rtf} are \emph{weakly similar}.

In Theorem \ref{vtd}(b), we prove that the   problem of classifying  pairs of commuting matrices up to weak similarity is wild, which ensures Theorem \ref{vtd}(a).

The analogous problem of classifying matrix pairs  $(A,B)$ up to weak congruence $S^T(\alpha A+\beta B,\gamma A+\delta B)S$
appears in the problem of classifying finite $p$-groups of nilpotency class 2 with commutator subgroup of type $(p,p)$, in the problem of classifying  commutative associative algebras with zero cube radical, and in the problem of classifying Lie algebras with central commutator subalgebra; see \cite{bel-dm,bel-lip,bro,ser_metab}.
The problem of classifying matrix pairs up to weak equivalence $R(\alpha A+\beta B,\gamma A+\delta B)S$ appears
in the theory of tensors \cite{bel-ber}.

Note that the group of $(n+1)\times (n+1)$ matrices
\[
\mat{A&v\\0&1},\qquad\text{in which }  A\in\FF^{n\times n}\text{ is nonsingular and}\ v\in\FF^n
\]
is called the \emph{general affine group}; it is the group of all invertible affine transformations of an affine space; see \cite{lyn}. If $\FF=\RR$, then this group is a Lie group, its Lie algebra consists of all $(n+1)\times (n+1)$ matrices
\[
\mat{A&v\\0&0},\qquad\text{in which }  A\in\RR^{n\times n}\text{ is nonsingular and}\ v\in\RR^n,
\]
and each Lie algebra  $\text{\rm L}(V)$ of matrices of the form \eqref{11q} with  $\FF=\RR$ is its subalgebra.

The abstract version of the construction of Lie algebras $\text{\rm L}(V)$  of matrices of the form
\eqref{11q} is the following. Let
$\FF[x,y]$ be the polynomial ring, and let $\mc W_{\FF[x,y]}$ be a left $\FF[x,y]$-module given by a finite dimensional vector space $\mc W_{\FF}$ and two commuting linear operators $P:w\mapsto xw$ and $Q:w\mapsto yw$ on $\mc W_{\FF}$ that are linearly independent. The $(2+\dim_{\FF}\mc W)$-dimensional vector space $L_{\mc W}:= \FF x\oplus_{\FF}\FF y\oplus_{\FF}\mc W$ is the metabelian Lie algebra with the bracket operation defined by $[x,v]:=Pv$, $[y,v]:=Qv$, and $[x,y]=[v,w]:=0$ for all $v,w\in \mc W$. If $\mc W=\FF^n$ and $V$ is the two-dimensional vector space generated by $P$ and $Q$, then the Lie algebra $L_{\mc W}$ coincides with the Lie algebra $\text{\rm L}(V)$ of all matrices \eqref{11q}.
By \cite[Corollary 1]{pet} and Theorem \ref{vtd}, \emph{the problem of classifying metabelian Lie algebras $L_{\mc W}$ is wild}.

We use the following definition of wild problems (see more formal definitions in \cite{bar,dro, gab}). Every matrix problem $\cal M$ is given
by a set ${\cal M}_1$ of tuples of
matrices over a field $\FF$ and a set ${\cal M}_2$ of
admissible transformations with them.
A matrix problem ${\cal M}$ is
{\it wild} if there exists a $t$-tuple
\begin{equation}\label{erf}
{M}(x,y)=(M_1(x,y),\dots,M_t(x,y))
\end{equation}
of matrices, whose entries are
noncommutative polynomials in
$x$ and $y$ over $\FF$, such that
\begin{itemize}
  \item[(i)]
${M}(A,B)
\in{\cal M}_1$ for all $A,B\in \FF^{n\times n}$ and $n=1,2,\dots$ (in particular, each scalar entry $\alpha $ of $M_i(x,y)$ is replaced by $\alpha I_n$),

  \item[(ii)]  ${M}(A,B)$ is reduced to ${M}(A',B')$ by transformations
${\cal M}_2$ if and only if $(A,B)$ is similar to $(A',B')$.
\end{itemize}

\section{Spaces of linear operators}\label{s2}

\begin{theorem}\label{vtd}
\begin{itemize}
  \item[{\rm(a)}] The problem of classifying up to similarity \eqref{qwa} of two-dimensional vector spaces of commuting matrices over a field\/ $\FF$ is wild. If\/ $\FF$ is not the field of two elements, then the problem of classifying up to similarity of two-dimensional vector spaces of commuting matrices over $\FF$ that contain nonsingular matrices is wild.

  \item[{\rm(b)}] The problem of classifying up to weak similarity \eqref{rtf} of pairs of commuting matrices over a field\/ $\FF$ is wild. If\/ $\FF$ is not the field of two elements, then the problem of classifying up to weak similarity of pairs $(A,B)$ of commuting matrices over\/ $\FF$ such that $\alpha A+\beta B$ is nonsingular for some $\alpha ,\beta \in\FF$ is wild.
\end{itemize}
\end{theorem}

\begin{proof} (a) This statement follows from statement (b) since each two-dimensional vector space $V\subset \FF^{n\times n}$ determined up to similarity is given by its basis $A,B\in V$ that is determined up to transformations \eqref{rtf}.

(b)
{\it Step 1: We prove that the problem of classifying pairs of commuting and nilpotent matrices up to similarity is wild.}  This statement was proved by  Gelfand and Ponomarev \cite{gel-pon}; it was extended in \cite{debora} to matrix pairs under consimilarity. By analogy with \cite[Section 3]{debora}, we consider two commuting and nilpotent $5n\times 5n$ matrices
\begin{equation}\label{gtr1}
J:=\left[ \begin{array}{cccc|c}
0&I_n&0&0&0\\0&0&I_n&0&0
\\0&0&0&I_n&0\\0&0&0&0&0\\
\hline 0&0&0&0&0
\end{array}  \right],
     \qquad
K_{XY}:=\left[ \begin{array}{cccc|c}
0&0&X&0&Y\\0&0&0& X&0\\0&0&0&0&0\\0&0&0&0&0\\
\hline 0&0&0&I_n&0
\end{array}  \right]
\end{equation}
that are partitioned into $n\times n$ blocks,
in which $X,Y\in\FF^{n\times n}$ are arbitrary.

 Let us prove that
\begin{equation}\label{njk1}
\parbox[c]{0.8\textwidth}{two pairs  $(X,Y)$ and $(X',Y')$ of $n\times n$ matrices are similar \ $\Longleftrightarrow$ \  two pairs of commuting and nilpotent matrices  $(J,K_{XY})$ and $(J,K_{X'Y'})$ are similar.}
\end{equation}

\noindent  ${\Longrightarrow\!\!.}$ If $(X,Y)S=S(X',Y')$, then $(J,K_{XY})R=R(J,K_{X'Y'})$ with $R:=S\oplus S\oplus S\oplus S\oplus S$.
\medskip

\noindent${\Longleftarrow\!\!.}$ Let $(J,K_{XY})R=R(J,K_{X'Y'})$ with nonsingular $R$. All matrices commuting with a given Jordan matrix are described in \cite[Section VIII, \S\,2]{gan}. Since $R$ commutes with $J$, we analogously find that
\[
R=\left[ \begin{array}{cccc|c}
C&C_1&C_2&C_3&D\\0&{C}&{C_1}
&{C_2}&0\\0&0&C&C_1&0\\0&0&0&{C}&0\\
\hline 0&0&0&E&F
\end{array}  \right].
\]
The equality
$K_{XY}R=RK_{X'Y'}$ implies that
\[
\left[ \begin{array}{cccc|c}
0&0&XC&XC_1+YE&YF
\\0&0&0&{X}{C}&0
\\0&0&0&0&0\\0&0&0&0&0\\
\hline 0&0&0&{C}&0
\end{array}  \right]
     =
\left[ \begin{array}{cccc|c}
0&0&{C}X'
&{C_1}{X'}+{D}
&{C}Y'
\\0&0&0&C{X'}&0
\\0&0&0&0&0\\0&0&0&0&0\\
\hline 0&0&0&{F}&0
\end{array}  \right],
\]
and so
$(X,Y)C={C}(X',Y')$.
\medskip

{\it Step 2: We prove that the problem of classifying matrix pairs up to weak similarity is wild.}
If the field $\FF$ has at least 3 elements, we fix any $\lambda \in\FF$ such that $\lambda\ne 0$ and $\lambda \ne -1$.   If $\FF$ consists of two elements, we take $\lambda =1$.

For each pair $(A,B)$ of $m\times m$ matrices with $m\ge 1$ over $\FF$,
define the matrix pair $(M_1(A),M_2(B))$ as follows:
\begin{align*}
M_1(A)&:=I_{2m+2}\oplus 0_{3m+3}\oplus I_{m+1}\oplus  A,\\
M_2(B)&:=0_{2m+2}\oplus I_{3m+3}\oplus \lambda I_{m+1}\oplus  B.
\end{align*}
(Analogous constructions are used in \cite{bel-dm,bel-lip}.)

Let us prove that  $(M_1(A),M_2(B))$  can be used in \eqref{erf} in order to prove the wildness of the problem of classifying matrix pairs up to weak similarity. We should prove that
\begin{equation}\label{njk}
\parbox[c]{0.8\textwidth}{arbitrary pairs  $(A,B)$ and $(A',B')$ of $m\times m$ matrices are similar \ $\Longleftrightarrow$ \ $(M_1(A),M_2(B))$ and $(M_1(A'),M_2(B'))$ are weakly similar.}
\end{equation}

\noindent  $\Longrightarrow\!\!.$ If $S^{-1}(A,B)S=(A',B')$,
then \[(I_{6m+6}\oplus S)^{-1}(M_1(A),M_2(B))(I_{6m+6}\oplus S)= (M_1(A'),M_2(B')).\]

\noindent$\Longleftarrow\!\!.$ Let
\[
S^{-1}(\alpha M_1(A)+\beta M_2(B),\gamma M_1(A)+\delta M_2(B))S=(M_1(A'),M_2(B'))
\]
with a nonsingular $\matt{\alpha&\beta \\\gamma &\delta }$.
Then
\begin{align*}
\rank(\alpha M_1(A)+\beta M_2(B))&=
\rank M_1(A'),\\
\rank(\gamma M_1(A)+\delta M_2( B))&=
\rank M_2(B').
\end{align*}

If $\beta \ne 0$, then \[\rank(\alpha M_1(A)+\beta M_2(B))> 4m+3\ge
\rank M_1(A').\] Hence $\beta =0$. Since $\matt{\alpha&\beta \\\gamma &\delta }$ is nonsingular, $\delta \ne 0$. If $\gamma \ne 0$, then \[\rank(\gamma M_1(A)+\delta M_2(B))> 5m+4\ge
\rank M_2(B').\] Hence $\gamma =0$.

Thus
\begin{equation*}\label{euc}
S^{-1}(\alpha M_1(A),\delta M_2(B))S=(M_1(A'),M_2(B')),
\end{equation*}
and so the pairs
\begin{equation}\label{666}
\begin{split}
(\alpha M_1(A),\delta M_2(B))&=(\alpha I_{2m+2}, 0_{2m+2})
\oplus (0_{3m+3},\delta I_{3m+3})\\ &
\qquad\qquad\qquad\oplus
(\alpha I_{m+1}, \delta\lambda I_{m+1})\oplus (\alpha A,\delta B),
\\
(M_1(A'),M_2(B'))&=(I_{2m+2}, 0_{2m+2})
\oplus (0_{3m+3},I_{3m+3})\\&\qquad\qquad\qquad\oplus
(I_{m+1},\lambda I_{m+1})\oplus (A',B')
\end{split}
\end{equation}
give isomorphic representations of the quiver
${\lefttorightarrow\!\!\ci\!\!\righttoleftarrow}$. By the Krull--Schmidt theorem for quiver representations (see \cite[Theorem 1.2]{sch}), every representation of a quiver is isomorphic to a direct sum of indecomposable representations, and this sum is uniquely
determined up to replacements of direct summands by isomorphic representations and permutations of direct
summands.

If we
delete in \eqref{666} the summands  $(\alpha I_{2m+2}, 0_{2m+2}) $ and $(0_{3m+3},\delta I_{3m+3})$ of $(\alpha M_1(A),\delta M_2(B))$  and the corresponding isomorphic summands $(I_{2m+2},0_{2m+2}) $ and $(0_{3m+3},I_{3m+3})$ of  $(M_1(A'),M_2(B'))$,  we find that the remaining pairs
\[
(\alpha I_{m+1}, \delta \lambda I_{m+1})\oplus (\alpha A,\delta B),\qquad (I_{m+1},\lambda I_{m+1})\oplus (A',B')
\]
give isomorphic representations of the quiver
${\lefttorightarrow\!\!\ci\!\!\righttoleftarrow}$.
The first pair has $m+1$ direct summands $(\alpha,\delta\lambda )$ and the second pair has $m+1$ direct summands $(1,\lambda )$. By the Krull--Schmidt theorem, these summands give isomorphic representations, hence $\alpha =\delta =1$, and so the pairs $(A,B)$ and $(A',B')$ give isomorphic representations too. Therefore, the pairs $(A,B)$ and $(A',B')$ are similar.

\medskip

{\it Step 3.}  By Steps 1 and 2, the following
equivalences hold for arbitrary pairs $(X,Y)$ and $(X',Y')$ of $n\times n$ matrices over $\FF$:
\begin{itemize}
  \item[\phantom{$\Longleftrightarrow$}]
$(X,Y)$ and $(X',Y')$ are similar

  \item[$\Longleftrightarrow$]
$(J,K_{XY})$ and $(J,K_{X'Y'})$ are similar

\item[$\Longleftrightarrow$]
$(\lambda I+J,K_{XY})$ and $(\lambda I+J,K_{X'Y'})$ are similar

\item[$\Longleftrightarrow$]
$(M_1(\lambda I+J),M_2(K_{XY}))$ and $(M_1(\lambda I+J),M_2(K_{X'Y'}))$ are weakly similar.
\end{itemize}
Note that $(M_1(\lambda I+J),M_2(K_{XY}))$ and $(M_1(\lambda I+J),M_2(K_{X'Y'}))$ are pairs of commuting matrices.
If $\FF$ has at least 3 elements, then the matrix
$M_1(\lambda I+J)+M_2(K_{XY})$ is nonsingular.
\end{proof}

\section{Lie algebras}\label{s3}

For each vector space $V\subset \FF^{n\times n}$ of commuting  matrices over a field $\FF$, we denote by $\widetilde V$ the vector space of all $(n+1)\times (n+1)$ matrices
of the form
\[
(A|a):=\begin{bmatrix}
  A & \begin{matrix}
        \alpha _1\\\vdots\\ \alpha _n \\
      \end{matrix}
   \\
  \begin{matrix}
        0&\dots&0 \\
      \end{matrix} & 0 \\
\end{bmatrix}, \qquad\text{in which  $A\in V$ and } a:=\mat{\alpha _1\\\vdots\\\alpha _n}\in\FF^n.
\]
We consider the space $\widetilde V$
as
the Lie algebra $\text{L}(V)$ with the Lie bracket operation
\begin{equation}\label{aew}
[(A|a),(B|b)]:=(A|a)(B|b)-(B|b)(A|a)=(0|Ab-Ba).
\end{equation}

\begin{theorem}\label{ttr}
Let a field\/ $\FF$ be not the field with 2 elements.
\begin{itemize}
  \item[\rm(a)]
Let $V\subset \FF^{n\times n}$ and $V'\subset \FF^{n'\times n'}$ be two vector spaces  of commuting  matrices that contain nonsingular matrices.
Then the following statements are equivalent:
\begin{itemize}
  \item[\rm(i)]
The Lie algebras $\text{\rm L}(V)$ and $\text{\rm L}(V')$ are
isomorphic.

  \item[\rm(ii)]
$n=n'$ and $V$ is similar to $V'$  $($i.e., $SVS^{-1}=V'$ for some nonsingular $S\in\FF^{n\times n})$,

  \item[\rm(iii)]
$n=n'$ and  $\widetilde V$ is similar to $\widetilde V'$.
\end{itemize}

  \item[\rm(b)] The problem of classifying Lie algebras $\text{\rm L}(V)$ with $\dim_{\FF} V=2$  up to isomorphism
is wild.
\end{itemize}
\end{theorem}

\begin{proof} (a) Let us prove the equivalence of (i)--(iii).

(i)$\Rightarrow$(ii) Let $\varphi :\text{L}(V)\widetilde{\to}\text{L}(V')$ be an isomorphism of Lie algebras. Then $\varphi [\widetilde V,\widetilde V]=[\widetilde V',\widetilde V']$.
By \eqref{aew}, $[\widetilde V,\widetilde V]\subset (0|\FF^n)$. Since $V$ contains a nonsingular matrix $A$, $[(A|0),(0|\FF^n)]=(0|\FF^n)$, and so $[\widetilde V,\widetilde V]= (0|\FF^n)$. Hence $\varphi (0|\FF^n)=(0|\FF^{n'})$ and $n=n'$.

Let $e_1,\dots,e_n$ be the standard basis of $\FF^n$, and let  $(0|f_i):=\varphi(0| e_i)$. Since $\varphi (0|\FF^n)=(0|\FF^{n})$,  $f_1,\dots,f_n$ is also a basis of $\FF^n$.
Denote by $S$ the nonsingular matrix whose columns are $f_1,\dots,f_n$. Then
\begin{equation}\label{33c}
f_i=Se_i.
\end{equation}

Let $A\in V$ and write $(B|b):=\varphi (A|0)$. Let $A=[\alpha _{ij}]_{i,j=1}^n$, i.e., $Ae_j=\sum_i\alpha _{ij}e_i$. Then
\begin{align*}
 (0|Bf_j)&=[(B|b),(0|f_j)]=[\varphi (A|0),\varphi (0|e_j)]=
 \varphi [(A|0),(0|e_j)]\\
&=\varphi (0|Ae_j)=\varphi (0,{\textstyle \sum\nolimits_i}\alpha _{ij}e_i)=\varphi {\textstyle \sum\nolimits_i}\alpha _{ij}(0|e_i)\\
&={\textstyle \sum\nolimits_i}\alpha _{ij}\varphi (0|e_i)=
{\textstyle \sum\nolimits_i}\alpha _{ij}(0|f_i)=
(0|{\textstyle \sum\nolimits_i}\alpha _{ij}f_i)
\end{align*}
and so $Bf_j=\sum_i\alpha _{ij}f_i$.
By \eqref{33c},
$$BSe_j={\textstyle \sum\nolimits_i}\alpha _{ij}Se_i=S{\textstyle \sum\nolimits_i}\alpha _{ij}e_i=SAe_j.$$ Therefore, $BS=SA$ and so $V'S=SV$.
\medskip

(ii)$\Rightarrow$(iii)
If $V$ and $V'$ are similar via $S$, then $\widetilde V$ and $\widetilde V'$ are similar via $S\oplus I_1$.
\medskip

(iii)$\Rightarrow$(i)
If $R\widetilde V R^{-1}=\widetilde V'$ for some nonsingular $R\in\FF^{(n+1)\times (n+1)}$, then
$X\mapsto RX R^{-1}$  is an isomorphism $\text{L}(V)\widetilde{\to}\text{L}(V')$.
\medskip

(b) This statement follows from the equivalence (i)$\Leftrightarrow$(ii) and Theorem 1(a).
\end{proof}

\section*{Acknowledgements}

V. Futorny  was
supported by  CNPq grant 301320/2013-6 and by
FAPESP grant 2014/09310-5.   V.V.~Sergeichuk was supported by
FAPESP grant 2015/05864-9.

\end{document}